\documentclass[11pt,reqno]{amsart}
\usepackage{cancel}
\usepackage{amsmath}
\usepackage{amsthm}
\usepackage{amssymb}
\usepackage{amsfonts,mathrsfs}
\usepackage{stmaryrd}
\usepackage{amsxtra}  
\usepackage{epsfig}
\usepackage{verbatim}
\usepackage{xcolor}
\usepackage[all]{xy}
\usepackage{graphicx}
\usepackage[displaymath, mathlines]{lineno}

\usepackage{tikz}

\theoremstyle{plain}
\newtheorem{theorem}[equation]{Theorem}
\newtheorem{proposition}[equation]{Proposition}
\newtheorem{lemma}[equation]{Lemma}
\newtheorem{corollary}[equation]{Corollary}

\theoremstyle{remark}

\newtheorem{remark}[equation]{Remark}
\numberwithin{equation}{section}

\newcommand{\re}{\text{Re}}
\newcommand{\im}{\text{Im}}

\newcommand{\sh}{{\mathscr H}}

\newcommand{\sL}{{\mathscr L}}

\newcommand{\C}{{\mathbb C}}

\newcommand{\N}{{\mathbb N}}

\newcommand{\R}{{\mathbb R}}

\begin{document}

\title[]{Plurisubharmonic defining functions in $\C^2$}
\author{Luka Mernik}
\address{Department of Mathematics, \newline Oklahoma State University, Stillwater, Oklahoma, USA}
\email{lmernik@okstate.edu}

\maketitle
 
\begin{abstract}

Let $\Omega=\{r<0\}\subset\C^2$, with $r$ plurisubharmonic on $b\Omega=\{r=0\}$. Let $\rho$ be another defining function for $\Omega$.
A formula for the determinant of the complex Hessian of $\rho$ in terms of $r$ is computed. This formula is used to give necessary and sufficient conditions that make $\rho$ (locally) plurisubharmonic.

\end{abstract}

\section{Introduction}

A domain $\Omega=\{r<0\}\subset \C^n$ is pseudoconvex if the complex Hessian of $r$ is positive semi-definite for all complex tangent vectors at all boundary points. A stronger property is admitting a plurisubharmonic defining function, that is, the complex Hessian of $r$ is positive semi-definite for all vectors in $\C^n$ at all points in $\Omega$. An intermediary condition is admitting a  defining function plurisubharmonic on the boundary, where the non-negativity of the complex Hessian need only occur at the boundary. 
Although every domain with a  defining function plurisubharmonic on the boundary is pseudoconvex, the converse is not true. Diederich and Fornaess \cite{DieFor77-1}, Fornaess \cite{Fornaess79}, and later Behrens \cite{Behrens85} found examples of weakly pseudoconvex domains in $\C^2$ which do not admit local plurisubharmonic defining functions, even on the boundary.  

The goal is to study the inequivalence of these three intraconnected notions of positivity. In other words, the aim is to understand the (in)ability of ``spreading'' of positivity of the complex Hessian to either non-tangent vectors or  points off the boundary.  Spreading of various kinds of positivity of the Hessian has been studied in \cite{HerMcN09} and  \cite{HerMcN2012}.

Since two kinds of spreading are involved, each is considered separately.
The first step is understanding the speading of the positivity of the complex Hessian from tangent vectors to the ``missing'' normal direction at boundary points. That is, answer the question whether a pseudoconvex domain admits a  defining function plurisubharmonic on the boundary. In \cite{Mernik20}, the author gives 
necessary and sufficient conditions for a pseudoconvex domain $\Omega$ to admit a local defining function plurisubharmonic on the boundary. The following expression for the determinant of the complex Hessian on the boundary was obtained
\begin{proposition}[\cite{Mernik20}]\label{Boundary}
Let $\Omega=\{r<0\}\subset \C^2$ with a smooth defining function $r$ and  let $\rho=r(1+Kr+T)$ be another defining function of $\Omega$. Then
\[\sh_\rho(p)=\sh_{(1+Kr+T)r}(p)=2Kh\sL_r(p)+\sh_{(1+T)r}(p) \text{ \quad for all $p\in b\Omega$,}\] 
where $\sL_r$ is the Levi form and $\sh_f$ is the determinant of the complex Hessian of $f$.
\end{proposition}

The second step involves spreading the positivity of the complex Hessian from the boundary and inside the domain. Namely, given a defining function plurisubharmonic on the boundary, does there exists a plurisubharmonic defining function and if so what modifications need to be made?
The goal of this paper is to generalize the Proposition \ref{Boundary}, by deriving a formula for $\sh_\rho$ that holds inside the domain as well. Let $\rho=r(1+Kr+X)=r(Kr+P)$. Then, as shown in Section \ref{S:Hessian},
\begin{align}\label{HessianIntro}
\sh_\rho=&\left(2KPH_r(L_r,L_r)+P^2\sh_r+2P\re[H_r(L_r,L_P)]+B_P\right)\notag\\
&+ r\bigg(4K^2H_r(L_r,L_r)+PQ_P+2\re[H_P(L_r,L_P)]+4KP\sh_r+4K\re[H_r(L_r,L_P)]\notag\\ 
&+2KH_P(L_r,L_r)\bigg)+r^2\left(4K^2\sh_r+\sh_P+2KQ_P\right),
\end{align}
where $H_f(V,W)$ is the complex Hessian of $f$ acting on vectors $V$ and $W$. The terms $B_P$ and $Q_P$ are  ``error'' terms to be defined later. These terms cannot, in particular, be written in terms of $H_f$ or $\sh_f$ for a relevant function $f$.

Under hypotheses of interest, many terms in \eqref{HessianIntro} can only be directly controlled on $b\Omega$. Taylor's formula, centered at a boundary point $p$ and used to compute $\sh_\rho$ at points $q\in \bar\Omega$ in the (real) normal direction from $p$, is the main analytical device used to pass information from $b\Omega$ into $\Omega$.

In Section \ref{S:pshonboundary} the Taylor expansion of $\sh_\rho$ is studied in greater detail. Assuming that
$\Omega$ admits a smooth  defining function plurisubharmonic on the boundary near $p\in b\Omega$ provides enough control on terms in \eqref{HessianIntro} to yield necessary and sufficient conditions on $X$ such that $\rho=r(1+Kr+X)$ is plurisubharmonic in a neighborhood of $p$. A difference between producing a plurisubharmonic defining function in a neighborhood of strongly and weakly pseudoconvex points is also observed. 
Section \ref{S:Examples} demonstrates the method through an example. 

In Section \ref{S:HOTS} higher order Taylor expansion is computed.

 A similar problem was considered by Liu \cite{Liu19-1},\cite{Liu19-2}. Recall that a Diederich-Fornaess exponent of $\Omega\subset\subset\C^n$ is a number $\eta\in(0,1]$ for which there exists a smooth defining function $\rho$ such that $-(-\rho)^\eta$ is strictly plurisubharmonic. Liu constructs an equation similar to  \eqref{HessianIntro} in order to control the size of such exponents. However factors of size $\frac 1{1-\eta}$ in the equation prevent its use in determining when $\Omega$ admits an actual plurisubharmonic defining function; or in other words has Diederich-Fornaess exponent exactly $1$. This is precisely the case detailed in this paper.

The author would like to thank Jeffery McNeal for numerous discussions about the topic of this paper and helpful comments while preparing this manuscript.

\section{Preliminaries}\label{S:Prel}

Notation and basic facts that used throughout the paper are recorded. Partial derivatives will be denoted with subscripts, e.g., $r_{z_j}=\frac{\partial r}{\partial z_j}$.
A defining function for $\Omega\subset \C^2$  is a function $r$ such that $\Omega=\{(z,w)\in \C^2: r(z,w)<0\}$ and $\nabla r\neq(0,0)$ on the boundary. 
Throughout the paper, $r$ is assumed to be smooth. However, smoothness can be relaxed to $C^2$ in most of the results.

If $p\in b\Omega$, translating coordinates reduces to considering $p=(0,0)$ is the origin. A further rotation produces 
\begin{align}\label{deffn}
r(z,w)=&\im w+F(z,w), \text{ for some real-valued $F$ with  } \\
F(0,0)=&0 \text{ and } \nabla F(0,0)=(0,0)\notag
\end{align}
Then $r_w(0,0)=\frac1{2i}$.

Let $H_f=\begin{pmatrix}
f_{z\bar z} & f_{z\bar w}\\
f_{\bar zw} & f_{w\bar w}
\end{pmatrix}$ denote the complex Hessian of $f$.  Denote $H_f$ acting on vectors $V=\langle V_1,V_2\rangle$ and $W=\langle W_1,W_2\rangle$  by
$$H_f(V,W)=VH_f\bar W=f_{z\bar z}V_1\bar W_1+f_{w\bar w}V_2\bar W_2+f_{z\bar w}V_1\bar W_2+f_{\bar zw}V_2\bar W_1.$$
The determinant of $H_f$ is denoted $\sh_f=\det H_f$.

Let $L_f=\frac{\partial f}{\partial w}\frac{\partial }{\partial z}- \frac{\partial f}{\partial z}\frac{\partial }{\partial w}$ and 
$N_f=\frac{\partial f}{\partial \bar z}\frac{\partial }{\partial z}+\frac{\partial f}{\partial \bar w}\frac{\partial }{\partial w}$. Then $L_r$ is the complex tangential and $N_r$ is the complex normal direction to the boundary. Furthermore, 
$$H_r(L_r,L_r)(p):=\sL_r(p)=r_{z\bar z}|r_w|^2+r_{w\bar w}|r_z|^2-2\re[r_{z\bar w}r_{\bar z}r_w]\big|_p$$ 
is the Levi form at the boundary point $p\in b\Omega$.

A function $f$ is plurisubharmonic if $H_f$ is a positive semi-definite matrix. By Sylvester's criterion \cite{Gilbert91}, $f$ is plurisubharmonic if $\sh_f=\det(H_f)\geq0$ and $f_{z\bar z},f_{w\bar w}\geq0$.

Big O notation is denoted by $\mathcal O$ with the asymptotics occuring at the origin, that is, $f(z)=\mathcal O(g(z))$ if there exists constants $M,\delta>0$ such that 
$$|f(z)|<M|g(z)|, \text{ when $0<|z|<\delta$.}$$

Finally, a version of Taylor's theorem  will be used extensively.
Since $b\Omega$ is smooth, there exists a neighborhood $U$ of $b\Omega$ and a smooth map
\begin{align*}
\pi:\bar\Omega\cap U\rightarrow& b\Omega\\
q\longmapsto& \pi(q)=p
\end{align*}
such that $p\in b\Omega$ lies on the (real) normal to $b\Omega$ passing through $q$. Let $d_{b\Omega}(q)$ be the complex euclidean distance of $q$ to $b\Omega$. Then $q=p-\frac{d_{b\Omega}(q)}{|\partial r(p)|}N_r(p)$.
Let $f\in C^2(\bar\Omega)$, $q\in \bar \Omega\cap U$, and $p=\pi(q)$. Taylor's formula in complex notation says
\begin{align*}
f(q)=&f(p)+ f_z(p)(q_1-p_1)+f_w(p)(q_2-p_2)+f_{\bar z}(p)(\bar q_1-\bar p_1)+f_{\bar w}(p)(\bar q_2-\bar p_2) +\mathcal O(d_{b\Omega}^2)\\
=&f(p) - \frac{d_{b\Omega}(q)}{|\partial r(p)|} \left[ r_{\bar z}(p)f_z(p)+r_{\bar w}(p) f_w(p) +r_z(p)f_{\bar z}(p)+r_w(p)f_{\bar w}(p)\right] +\mathcal O(d_{b\Omega}^2)\\
=&f(p) -  2\frac{d_{b\Omega}(q)}{|\partial r(p)|}[(\re N)(f)](p)+\mathcal O(d_{b\Omega}^2)
\end{align*}
Since $-\frac{d_{b\Omega}}{|\partial r|}$ is another defining function for $\Omega$, there exists a positive real-valued function $u$ such that 
$-\frac{d_{b\Omega}}{|\partial r|}=u\cdot r$.
Therefore Taylor's formula can be written as
\begin{align}\label{Taylor}
f(q)=f(p) +2u(q) r(q) [(\re N_r)f](p) +\mathcal O(r^2).
\end{align}
If $f$ is real-valued,  \eqref{Taylor} becomes
$$f(q)=f(p)+2u(q)r(q)\re[N_rf](p)+\mathcal O(r^2).$$

\section{Determinant of the complex Hessian}\label{S:Hessian}

An arbitrary defining function for $\Omega$ is necessarily a multiple of $r$, i.e.,  $\rho=r\cdot h$ for some real-valued positive function $h$.
By rescaling write $$h=1+Kr+X$$ for $K\in\R$ and $X$ a real-valued function with $X(0,0)=0$. This decomposition is not unique, but  we are interested in properties $X$ needs to satisfy 
so that $\rho=r(1+Kr+X)$ is plurisubharmonic.
For brevity write $P=1+X$. Note that $P>0$ in a sufficiently small neighborhood of the origin.

In this section the  determinant of the complex Hessian of $\rho$ is computed in terms of $r$ and $P$. This formula is the basis for most of the simplifications in this paper.
\begin{align}\label{Hessian}
\sh_\rho=&\rho_{z\bar z}\rho_{w\bar w} - |\rho_{z\bar w}|^2\notag\\
=& \left((Kr^2)_{z\bar z}+(Pr)_{z\bar z}\right)\left((Kr^2)_{w\bar w}+(Pr)_{w\bar w}\right) - \left|(Kr^2)_{z\bar w}+(Pr)_{z\bar w}\right|^2\notag\\
=& (Kr^2)_{z\bar z}(Kr^2)_{w\bar w}+ (Pr)_{z\bar z}(Kr^2)_{w\bar w}+(Pr)_{w\bar w}(Kr^2)_{z\bar z}+(Pr)_{z\bar z}(Pr)_{w\bar w} \notag\\
&-|(Kr^2)_{z\bar w}|^2-|(Pr)_{z\bar w}|^2 -2\re[(Kr^2)_{z\bar w}(Pr)_{\bar zw}]\notag\\
=&K^2\sh_{r^2} + \sh_{Pr} +\underbrace{K\left((Pr)_{z\bar z}(r^2)_{w\bar w}+(Pr)_{w\bar w}(r^2)_{z\bar z}-2\re[(r^2)_{z\bar w}(Pr)_{\bar zw}]\right)}_{A}.
\end{align}

Consider each term in \eqref{Hessian} separately and organize them in terms of powers of $r$.
The first term is
\begin{align}\label{HessianRR}
K^2 \sh_{r^2}=& K^2 (r^2)_{z\bar z}(r^2)_{w\bar w}-|(r^2)_{z\bar w}|^2\notag\\
=&(2rr_{z\bar z}+2|r_z|^2)(2rr_{w\bar w}+2|r_w|^2)-|2rr_{z\bar w}+2r_zr_{\bar w}|^2\notag\\
=&4K^2(r^2r_{z\bar z}r_{w\bar w} +rr_{z\bar z}|r_w|^2 +rr_{w\bar w}|r_z|^2+|r_z|^2|r_w|^2\notag\\
& \qquad\quad- r^2|r_{z\bar w}|^2-|r_z|^2|r_w|^2-2\re[rr_{z\bar w}r_{\bar z}r_w] )\notag\\
=&r(4K^2 H_r(L_r,L_r))+r^2(4K^2\sh_{r}).
\end{align}

The second term is
\begin{align}\label{HessianPR}
\sh_{Pr}=&(Pr)_{z\bar z}(Pr)_{w\bar w}-|(Pr)_{z\bar w}|^2   \notag\\
=& (Pr_{z\bar z}+2\re[r_zP_{\bar z}]+rP_{z\bar z})(Pr_{w\bar w} +2\re[r_wP_{\bar w}]+rP_{w\bar w})\notag\\
&-|Pr_{z\bar w}+r_zP_{\bar w}+r_{\bar w}P_z+rP_{z\bar w}|^2\notag\\
=& P^2r_{z\bar z}r_{w\bar w}+2Pr_{z\bar z}\re[r_wP_{\bar w}]+Prr_{z\bar z}P_{w\bar w} 
+2Pr_{w\bar w}\re[r_zP_{\bar z}]+4\re[r_zP_{\bar z}]\re[r_wP_{\bar w}]\notag\\ 
&+2rP_{w\bar w}\re[r_zP_{\bar z}] +Prr_{w\bar w}P_{z\bar z}+2rP_{z\bar z}\re[r_wP_{\bar w}] + r^2P_{z\bar z}P_{w\bar w}
- P^2|r_{z\bar w}|^2 - r^2|P_{z\bar w}|^2\notag\\
&-|r_zP_{\bar w}+r_{\bar w}P_z|^2 -2\re[Pr_{z\bar w}(r_{\bar z}P_w+r_wP_{\bar z})] -2\re[Prr_{z\bar w}P_{\bar zw}] \notag\\
&-2\re[rP_{z\bar w}(r_{\bar z}P_w+r_wP_{\bar z})]\notag\\
=& P^2\left(r_{z\bar z}r_{w\bar w}- |r_{z\bar w}|^2\right) + 2P\left(r_{z\bar z}\re[r_wP_{\bar w}]+r_{w\bar w}\re[r_zP_{\bar z}]-\re[r_{z\bar w}(r_{\bar z}P_w+r_wP_{\bar z})]\right) \notag\\
&+2r\left(P_{w\bar w}\re[r_zP_{\bar z}]+P_{z\bar z}\re[r_wP_{\bar w}]-\re[P_{z\bar w}(r_{\bar z}P_w+r_wP_{\bar z})]\right) 
+r^2\left(P_{z\bar z}P_{w\bar w}-|P_{z\bar w}|^2\right)\notag\\
&+Pr(\underbrace{r_{z\bar z}P_{w\bar w}+r_{w\bar w}P_{z\bar w}-2\re[r_{z\bar w}P_{\bar zw}]}_{Q_P}) 
+\underbrace{4\re[r_zP_{\bar z}]\re[r_wP_{\bar w}]-|r_zP_{\bar w}+r_{\bar w}P_z|^2}_{B_P} \notag\\
=&\left(P^2\sh_r+ 2P\re[H_r(L_r,L_P)] +B_P\right)+r\left(PQ_P+2\re[H_P(L_r,L_P)]\right)+r^2\left (\sh_P\right) .
\end{align}

And finally
\begin{align}\label{HessianA}
A=&K\left((Pr)_{z\bar z}(r^2)_{w\bar w}+(Pr)_{w\bar w}(r^2)_{z\bar z}-2\re[(Pr)_{\bar zw}(r^2)_{z\bar w}]\right) \notag\\
=&K\big((Pr_{z\bar z}+2\re[r_zP_{\bar z}]+rP_{z\bar z})(2rr_{w\bar w}+2|r_w|^2)\notag\\
&+(Pr_{w\bar w}+2\re[r_wP_{\bar w}]+rP_{w\bar w})(2rr_{z\bar z}+2|r_z|^2)\notag\\
&-2\re[(Pr_{\bar zw}+r_{\bar z}P_w+r_wP_{\bar z}+rP_{\bar zw})(2rr_{z\bar w}+2r_zr_{\bar w})]\big) \notag\\
=&2K\big( Prr_{z\bar z}r_{w\bar w}+Pr_{z\bar z}|r_w|^2+2rr_{w\bar w}\re[r_zP_{\bar z}] +\cancel{2|r_w|^2\re[r_zP_{\bar z}]}
+r^2r_{w\bar w}P_{z\bar z}+rP_{z\bar z}|r_w|^2 \notag\\
&+Prr_{w\bar w}r_{z\bar z}+Pr_{w\bar w}|r_z|^2+2rr_{z\bar z}\re[r_wP_{\bar w}]+\cancel{2|r_z|^2\re[r_wP_{\bar w}]} 
+r^2r_{z\bar z}P_{w\bar w}+rP_{w\bar w}|r_z|^2 \notag\\
&-2\re[Pr|r_{z\bar w}|^2]-2\re[Pr_{\bar z w}r_zr_{\bar w}]-2\re[rr_{z\bar w}(r_{\bar z}P_w+r_wP_{\bar z})] \notag\\
&\cancel{- 2\re[|r_z|^2P_wr_{\bar w}]}\cancel{-2\re[|r_w|^2r_zP_{\bar z}]}-2\re[r^2r_{z\bar w}P_{\bar zw}]-2\re[rP_{\bar zw}r_zr_{\bar w}]\big)\notag\\
=&2K\big( Pr\left(2r_{z\bar z}r_{w\bar w}-2|r_{z\bar w}|^2\right)+P\left(r_{z\bar z}|r_w|^2+r_{w\bar w}|r_z|^2 -2\re[r_{\bar z w}r_zr_{\bar w}]\right) \notag\\
&+ 2r\left(r_{w\bar w}\re[r_zP_{\bar z}]+r_{z\bar z}\re[r_wP_{\bar w}] -2\re[r_{z\bar w}(r_{\bar z}P_w+r_wP_{\bar z})]\right) \notag\\
&+r^2\left(r_{w\bar w}P_{z\bar z}+r_{z\bar z}P_{w\bar w}-2\re[r_{z\bar w}P_{\bar zw}]\right)
+r\left(P_{z\bar z}|r_w|^2+P_{w\bar w}|r_z|^2-2\re[P_{\bar z w}r_zr_{\bar w}]\right)\big)\notag\\
=&2KPH_r(L_r,L_r) + r\left(4KP\sh_r+4K\re[H_r(L_r,L_P)]+2KH_P(L_r,L_r)\right) +r^2\left(2KQ_P\right).
\end{align}

Substituting \eqref{HessianRR}, \eqref{HessianPR}, and \eqref{HessianA} into \eqref{Hessian},
\begin{align}\label{Hessian2}
\sh_\rho=&\left(2KPH_r(L_r,L_r)+P^2\sh_r+2P\re[H_r(L_r,L_P)]+B_P\right)\notag\\
&+ r\big(4K^2H_r(L_r,L_r)+PQ_P+2\re[H_P(L_r,L_P)]+4KP\sh_r+4K\re[H_r(L_r,L_P)]\notag\\ 
&+2KH_P(L_r,L_r)\big)+r^2\left(4K^2\sh_r+\sh_P+2KQ_P\right).
\end{align}

Now apply \eqref{Taylor} to each relevant term in \eqref{Hessian2} up to power $r^2$
\begin{align}\label{Hessian3}
\sh_\rho(q)=&2KP(q)H_r(L_r,L_r)(p)+4KP(q)r(q)u(q)\re[N_rH_r(L_r,L_r)](p)\notag\\
&+P^2(q)\sh_r(p)+2P^2(q)r(q)u(q)\re[N_r\sh_r](p) \notag\\
&+2P(q)\re[H_r(L_r,L_P)](p)+4P(q)r(q)u(q)\re[N_r\re[H_r(L_r,L_P)]](p)\notag\\
&+B_P(p)+2r(q)u(q)\re[N_rB_P](p)\notag\\
&+r(q)\big(4K^2 H_r(L_r,L_r)(p)+ P(q)Q_P(p)+2\re[H_P(L_r,L_P)](p)+4KP(q)\sh_r(p)\notag\\
&+4K\re[H_r(L_r,L_P)](p)+2KH_P(L_r,L_r)(p)\big) +\mathcal O(r^2).
\end{align}
For brevity, drop the point $q$ from notation. Recall that $H_r(L_r,L_r)(p)=\sL_r(p)$. Combining like powers of $r$ in \eqref{Hessian3}
\begin{align}\label{HessianFinal}
\sh_\rho(q)=& \bigg(2KP\sL_r(p)+P^2\sh_r(p)+2P\re[H_r(L_r,L_P)](p)+B_P(p)\bigg) \notag\\
+r&\bigg(4KPu\re[N_rH_r(L_r,L_r)](p)+2P^2u\re[N_r\sh_r](p) \notag\\
&+4Pu\re[N_r\re[H_r(L_r,L_P)]](p)+2u\re[N_rB_P](p) \notag\\
&+4K^2\sL_r(p)+PQ_P(p)+2\re[H_P(L_r,L_P)](p)+4KP\sh_r(p) \notag\\
&+4K\re[H_r(L_r,L_P)](p)+2KH_P(L_r,L_r)(p)\bigg)+\mathcal O(r^2)
\end{align}
is obtained. For clarity: all functions in \eqref{HessianFinal} not explicitly evaluated are evaluated at $q$.

\section{Domains with a local plurisubharmonic defining function on the boundary}\label{S:pshonboundary}

Now suppose $\Omega\subset \C^2$ admits a defining function $r$ that is plurisubharmonic on $b\Omega$. 
Then $H_r(V,V)(p)\geq0$ for all vectors $V\in \C^2$ and $p\in b\Omega$; in particular 
$$H_r(L_r,L_r)(p)=\sL_r(p)\geq0.$$ Say that $p$ is a weakly pseudoconvex point if 
$\sL_r(p)=0$ and a strongly pseudoconvex point if $\sL_r(p)>0$. Strongly and weakly pseudoconvex will be considered separately.

Let $$\rho=r(1+Kr+X).$$ The objective is to find conditions on functions $P=1+X$ that make $\rho$ plurisubharmonic. Computing $\rho_{w\bar w}$ gives
\begin{align}
\rho_{w\bar w}=&(1+Kr+X)r_{w\bar w}+2\re[r_w(Kr_{\bar w}+X_{\bar w})]+r(Kr_{w\bar w}+X_{w\bar w}),\notag
\end{align}
and evaluation at the origin yields
\begin{align}
\rho_{w\bar w}(0,0)=&r_{w\bar w}(0,0) +2K|r_w(0,0)|^2+2\re[r_w(0,0)X_{\bar w}(0,0)] \notag\\
=&r_{w\bar w}(0,0)+\frac{K}{2} +2\re[\frac1{2i}X_{\bar w}(0,0)]\notag.
\end{align}
For any $C>0$, $\rho_{w\bar w}(0,0)>2C>0$ if $K>0$ is chosen big enough. Therefore 
\begin{align}\label{RHOww}
\rho_{w\bar w}>C>0
\end{align}
 in a sufficiently small neighborhood of the origin, if $K>0$ large enough. Consequently, focus can be turned to making $\sh_\rho\geq0$.

Consider the constant terms (with respect to $r$) in \eqref{HessianFinal}. Define $$G_{\rho}(q):=2KP(q)\sL_r(p)+P^2(q)\sh_r(p)+2P(q)\re[H_r(L_r,L_P)](p)+B_P(p)$$ for $q\in\bar \Omega$ with $\pi(q)=p$. Then $$\sh_\rho(q)=G_{\rho}(q)+\mathcal O(r).$$

\begin{proposition}\label{P:N1}
Suppose that $\rho$ is plurisubharmonic. 
Then \begin{enumerate}\item $G_\rho(p)\geq0$ for all boundary points $p\in b\Omega$, and 
\item  if $G_\rho(p_0)>0$ for $p_0\in b\Omega$, then $\rho$ is plurisubharmonic in a neighborhood of $p_0$.
\end{enumerate}
\end{proposition}

\begin{proof}
By assumption, $\rho$ is plurisubharmonic if $p\in b\Omega$. Evaluating \eqref{HessianFinal} at $p\in b\Omega$ yields
\[0\leq \sh_\rho(p)=2KP(p)\sL_r(p)+P^2(p)\sh_r(p)+2P(p)\re[H_r(L_r,L_P)](p)+B_P(p)=G_\rho(p). \]

For (2) assume $G_\rho(p_0)>0$. Then $\sh_{\rho}(p_0)>0$ and continuity shows there exists a neighborhood $U$ of $p_0$ such that $\sh_\rho(q)>0$ for all $q\in U$. 
Since $\sL_r(p_0)\geq0$ and $P(p_0)>0$, increasing $K>0$ will not affect $\sh_\rho\geq 0$.
With \eqref{RHOww}, this shows $\rho$ is plurisubharmonic for $K>0$ large enough, in a sufficiently small neighborhood of $p_0$.
\end{proof}

\subsection{Strongly pseudoconvex points}
\begin{corollary}\label{C:S1}  Let $\Omega=\{r<0\}\subset \C^2$ with a  defining function plurisubharmonic on the boundary $r$. Let $p\in b\Omega$ be a strongly pseudoconvex point. Then there exists a neighborhood $U$ of $p$ such that $r$ is plurisubharmonic at every point in $U$. In particular, $r$ is plurisubharmonic on $\{q\in U\cap\Omega: \pi(q)=p\}$.
\end{corollary}

\begin{proof}

Since $r$ is plurisubharmonic and $p$ is strongly pseudoconvex point, $r$ is strongly plurisubharmonic at $p$, that is, $\sh_r(p)>0$. The result follows from continuity. 

\end{proof}

\begin{remark}
Given enough positivity of $G_\rho(p)-2KP(p)\sL_r(p)$ and $\rho_{w\bar w}(p)$,  $K<0$ can be chosen as well. 
\end{remark}

Corollary \ref{C:S1} shows that for points $q\in \Omega$ with $\pi(q)=p$ where $p$ is a strongly pseudoconvex point, the determinant of the complex Hessian can be always made positive in a small neighborhood of $p$ no matter the choice of a real-valued function $X$. 

This recovers, from a different viewpoint, the $n=2$ case of a result of Kohn:
\begin{theorem}[Kohn \cite{Kohn73}]  Let $\Omega=\{r<0\}\subset \C^n$ be a strongly pseudoconvex domain with a defining function $r$. 
Then $\rho=r(1+Kr)$ is a plurisubharmonic defining function for some $K>0$.
\end{theorem}

\subsection{Weakly pseudoconvex points}
The difficulty of producing a plurisubharmonic defining function occurs at points $q\in \Omega$ which lie in the normal direction from weakly pseudoconvex points. From now on let $p\in b\Omega$ be a weakly pseudoconvex point and let $q\in \bar \Omega$ with $\pi(q)=p$. Let $$W=\{p\in b\Omega: \sL_r(p)=0\}$$ be a set of weakly pseudoconvex points of $\Omega$. 

First we collect some basic facts about plurisubharmonic defining functions at weakly pseudoconvex points.
\begin{lemma}\label{L:CS}
Suppose that $r$ is plurisubharmonic on the boundary and $p\in b\Omega$ is weakly pseudoconvex. Then for all vectors $V\in \C^2$
$$H_r(L_r,V)(p)=0.$$ In particular $H_r(L_r,L_P)=0$.
\end{lemma}
\begin{proof}
Since $H_r$ is positive semi-definite Cauchy-Schwarz applies 
$$|H_r(L_r,V)(p)|\leq(H_r(L_r,L_r)(p))^{\frac12}(H_r(V,V)(p))^{\frac12}.$$
The conclusion follows since $H_r(L_r,L_r)(p)=\sL_r(p)=0$ for weakly pseudoconvex points.
\end{proof}

\begin{lemma}\label{L:SH}
Suppose that $r$ is plurisubharmonic on the boundary and $p\in b\Omega$. Then:
\begin{enumerate} 
\item $\sh_r(p)\geq0, \text{ and }$
\item if $p$ is a weakly pseudoconvex point then $\sh_r(p)=0.$
\end{enumerate}
\end{lemma}

\begin{proof}
(1) is immediate from Sylvester's criterion. If $p$ is weakly pseudoconvex $0$ is an eigenvalue of $H_r$ as $H_r(L_r,L_r)(p)=0$. Therefore, $H_r$ cannot be positive definite and by Sylvester's criterion $\sh_r(p)\not>0$.
\end{proof}

Starting with equation \eqref{HessianFinal} and using Lemma \ref{L:CS} and Lemma \ref{L:SH}
\begin{align}\label{HessianWeak}
\sh_\rho(q)= B_P(p) + r\bigg(&4KPu\re[N_rH_r(L_r,L_r)](p)+2P^2u\re[N_r\sh_r](p)\notag\\
&+4Pu\re[N_r\re[H_r(L_r,L_P)]](p)  +2u\re[N_rB_P](p)+PQ_P(p) \notag\\
&+2\re[H_P(L_r,L_P)](p)+2K H_P(L_r,L_r)(p)\bigg)+\mathcal O(r^2)
\end{align} 
for all $p\in W$.

Examining each power of $r$ in \eqref{HessianWeak} a necessary condition as well as a sufficient condition is obtained.
Considering the constant terms in \eqref{HessianWeak} gives a necessary condition:

\begin{lemma}\label{P:NWeak}
Suppose that $\rho$ is plurisubharmonic and $p\in W$. Then
$$B_P(p)= (4\re[r_zP_{\bar z}]\re[r_wP_{\bar w}]-|r_zP_{\bar w}+r_{\bar w}P_z|^2)\big|_p=0.$$
\end{lemma}

\begin{proof}
Since $p\in b\Omega$ is a weakly pseudoconvex point, $$0=\sh_{\rho}(p)=B_P(p).$$
In fact, Lemma \ref{BPbound} shows that $B_P\leq0$. In order to preserve plurisubharmonicity on the boundary $B_P(p)=0$.
\end{proof}

Proposition \ref{P:NWeak} shows the differential equation $B_P=0$ and the set $Z(B_P)=\{p\in b\Omega: B_P(p)=0\}$ are of critical importance.  In particular, only real-valued functions $P$ with $W\subset Z(B_P)$ need be considered. Some elementary lemmas about $B_P$ and $Z(B_P)$ are now derived, aiming towards building a library of functions $P$ satisfying $W\subset Z(B_P)$. It is clear that if $P=1+X$, $B_P=B_X$ and so $P$ and $X$ are used interchangeably.

\begin{lemma}\label{BPbound} 
Let $p\in b\Omega$ be any (not necessarily weakly pseudoconvex point) point in the boundary of $\Omega$. Then
$$B_P(p)\leq 0 \text{ and }$$
$$B_P(p)=0 \text{ if and only if } L_r(P)(p)=0.$$
\end{lemma}
\begin{proof}
Expanding the definition of $B_P$ gives us
\begin{align}\label{E:BPdef}
B_P=& 4\re[r_zP_{\bar z}]\re[r_wP_{\bar w}]-|r_zP_{\bar w}+r_{\bar w}P_z|^2\notag\\
=& 4\re[r_zP_{\bar z}]\re[r_wP_{\bar w}]-|r_z|^2|P_w|^2-|r_w|^2|P_z|^2-2\re[r_zP_{\bar w}r_wP_{\bar z}]\notag\\
=& 4\re[r_zP_{\bar z}]\re[r_wP_{\bar w}]-|r_z|^2|P_w|^2-|r_w|^2|P_z|^2\notag\\
&-2\re[r_zP_{\bar z}]\re[r_wP_{\bar w}] + 2\im[r_zP_{\bar z}]\im[r_wP_{\bar w}]\notag\\
=& -|r_z|^2|P_w|^2-|r_w|^2|P_z|^2 +2\re[r_zP_{\bar z}]\re[r_wP_{\bar w}]+2\im[r_zP_{\bar z}]\im[r_wP_{\bar w}] 
\end{align}
Using the Arithmetic-Geometric mean inequality
\begin{equation}\label{E:BPAMGM} 
|r_z|^2|P_w|^2+|r_w|^2|P_z|^2\geq2\sqrt{|r_z|^2|P_w|^2|r_w|^2|P_z|^2}=2|r_z||P_w||r_w||P_z| .
\end{equation}
By Cauchy-Schwarz
\begin{align}\label{E:BPCS}
&(\re[r_zP_{\bar z}]\re[r_wP_{\bar w}]+\im[r_zP_{\bar z}]\im[r_wP_{\bar w}])\notag\\
\leq& ((\re[r_zP_{\bar z}])^2+(\im[r_zP_{\bar z}])^2)^\frac12
((\re[r_wP_{\bar w}])^2+(\im[r_wP_{\bar w}])^2)^\frac12\notag\\
=&(|r_zP_z|^2)^\frac12(|r_wP_w|^2)^\frac12 =|r_z||P_z||r_w||P_w|.
\end{align}
Substituting inequalities \eqref{E:BPAMGM} and \eqref{E:BPCS} into equation \eqref{E:BPdef} proves the result.
Furthermore, the equality holds if and only if 
\begin{align}\label{E:equality1}
|r_z||P_w|=|r_w||P_z|
\end{align} and 
\begin{align}\label{E:equality}
\langle\re[r_zP_{\bar z}] ,\im[r_zP_{\bar z}]\rangle=\lambda \langle\re[r_wP_{\bar w}] ,\im[r_wP_{\bar w}]\rangle
\end{align} for some $\lambda\in \R$. However, notice that if $\lambda<0$ both terms in the definition of  $B_P$ are non-positive and must both equal $0$ for the equality to hold. Thus, we may assume $\lambda\geq0$.

The equality \eqref{E:equality} can be rephrased as following:
\begin{align}\label{E:equality2}
r_zP_{\bar z}=\re[r_zP_{\bar z}]+i\im[r_zP_{\bar z}]=\lambda\re[r_wP_{\bar w}]+i\lambda\im[r_wP_{\bar w}]=\lambda r_wP_{\bar w}.
\end{align}

If any of $r_z,r_w,P_z,P_w$ equal $0$ at $p$, the equation \eqref{E:equality1} says that $$L_r(P)(p)=r_w(p)P_z(p)-r_z(p)P_w(p)=0.$$

Now suppose none of the terms vanish at $p$.
Then taking the modulus of each side of \eqref{E:equality2} gives $|r_z||P_z|=\lambda|r_w||P_w|$. Solving for $|P_w|=\frac{|r_z||P_z|}{\lambda|r_w|}$, substituting it into equation \eqref{E:equality1}, and solving for $\lambda$
\begin{align*}
\lambda =&\frac{|r_z|^2}{|r_w|^2}.
\end{align*}

Finally, substituting $\lambda$ back into \eqref{E:equality2}
\begin{align*}
0=&r_zP_{\bar z}-\lambda r_wP_{\bar w}=r_zP_{\bar z}-\frac{|r_z|^2}{|r_w|^2}r_wP_{\bar w}\\
=& \frac{r_z}{r_{\bar w}}\left( r_{\bar w}P_{\bar z} - r_{\bar z}P_{\bar w}\right) = \frac{r_z}{r_{\bar w}} \overline{L_r(P)} .
\end{align*}
The above string of equalities show that \eqref{E:equality1} and \eqref{E:equality2} is equivalent to $L_r(P)(p)=0$.
\end{proof}

\begin{lemma}\label{BPscale}
For all $\alpha\in \R\setminus \{0\}$, $$Z(B_{\alpha P})=Z(B_P).$$
\end{lemma}
\begin{proof}
The result follows from homogeneity of $B_P$,
\begin{align*}
B_{\alpha P}=&4\re[r_z(\alpha P_{\bar z})]\re[r_w(\alpha P_{\bar w})]-|r_z(\alpha P_{\bar w})+r_{\bar w}(\alpha P_z)|^2\\
=&4\alpha^2\big(\re[r_zP_{\bar z}]\re[r_wP_{\bar w}]-|r_zP_{\bar w}+r_{\bar w}P_z|^2\big)\\
=&4\alpha^2B_P.
\end{align*}
\end{proof}

\section{Necessary and Sufficient Conditions}\label{NSC} 

\subsection{Necessary Conditions}

In order to guarantee that $\rho$ remains plurisubharmonic near $p_0$, we use the following theorem from \cite{Mernik20}:
\begin{theorem}\label{BoundaryPSH}
A domain $\Omega$ admits a local defining function which is plurisubharmonic on the boundary if and only if there exists a real valued function $T$ vanishing at $p_0$ and constant $C>0$ such that $-\sh_{(1+X)r}\leq C\sL_r$ in a neighborhood of $p_0$. 

Furthermore $\rho=(1+Kr+X)$ is a local defining function plurisubharmonic on the boundary near $p_0$ for some $K$.
\end{theorem}
Thus the first necessary condition is $\sh_{(1+X)r}\leq C\sL_r$ for some $C>0$.
Furthermore, Lemma \ref{BPbound} gives the second necessary condition $$W\subset Z(L_r(P))=Z(L_r(1+X)) .$$

\subsection{Sufficient Conditions}	

 By considering terms in the coefficient of $r$ in \eqref{HessianWeak}, the following sufficient condition for making $\rho$ plurisubharmonic is  obtained:
\begin{proposition}\label{Suff}
Let $W$ be the set of weakly pseudoconvex points of $\Omega$. Let $(0,0)=p_0\in W$ be the origin and let $U$ be a neighborhood of $p_0$. Suppose that there exists a real-valued function $P=1+X$ such that
\begin{enumerate} 
\item $W\cap U\subset Z(B_P)$, i.e., for all $p\in W\cap U$ $L_r(P)(p)=0$,
\item $H_P(L_r,L_r)(p_0)>0$, and
\end{enumerate}
Then there exist  $L_0$ and $K_0(L_0)$ such that such that for all $L< L_0$ and $K\geq K_0$ $\rho=r(1+Kr+LX)$ is plurisubharmonic at points lying on a normal line from weakly pseudoconvex points in some neighborhood of $p_0$.
\end{proposition}

\begin{proof}
Suppose that there exists a real-valued function $P=1+X$ that satisfies (1) and (2).

Let $p\in W\cap U$ and $q\in \Omega$ with $\pi(q)=p$.

By continuity, there exists $\epsilon>0$ and a neighborhood $U_1\subset U$ of $p_0$ such that 
$$H_{1+X}(L_r,L_r)(p)>\delta>0, \text { for all $p\in U_1$.} $$

Starting with \eqref{HessianWeak} and using the assumption (1)
\begin{align}\label{HessianWeak1}
\sh_\rho(q)=r\bigg(&4KPu\re[NH_r(L_r,L_r)](p)+2P^2u\re[N\sh_r](p)\notag\\
&+4Pu\re[N_r\re[H_r(L_r,L_P)]](p)  +2u\re[N_rB_P](p)+PQ_P(p) \notag\\
&+2\re[H_P(L_r,L_P)](p)+2K H_P(L_r,L_r)(p)\bigg)+\mathcal O(r^2) .
\end{align}
In a sufficiently small neighborhood $U_2\subset U_1$ of $p_0$
\begin{align*}
2P^2u\re[N_r\sh_r](p)+4Pu\re[N_r\re[H_r(L_r,L_P)]](p)+2u\re[N_rB_P](p)&\\
+PQ_P(p)+2\re[H_P(L_r,L_P)](p)+2K H_P(L_r,L_r)(p)&\leq C_1
\end{align*}
is bounded and $C_1$ is independent of $K$ and
\[Pu\re[NH_r(L_r,L_r)](p)\leq C_2\] is bounded.
Picking $L\leq-\frac{(2+\epsilon)C_2}{\delta}$, where $\epsilon>0$. Recalling $r(q)<0$
\begin{align}
\sh_\rho(q)\geq& r( 4KC_2 +C_1 + 2K L H_{1+X}(L_r,L_r)(p))+\mathcal O(r^2) \notag\\
>& r(4KC_2 +C_1 -(4K+\epsilon K)C_2)+\mathcal O(r^2)\notag\\
\geq&r(-\epsilon KC_2+2C_1)\notag
\end{align}
in a sufficiently small neighborhood.
Finally, for $K\geq\frac{C_1}{\epsilon C_2}>0$
\[\sh_\rho(q)\geq -2C_1r(q)\geq0 \text{ for all $q\in U_2$ with $\pi(q)=p$ and $p\in W\cap U_2$} .\]
This shows that $\rho$ is plurisubharmonic at points lying on a normal line from weakly pseudoconvex points in some neighborhood of $p_0$.
\end{proof}

The assumption (2) in Proposition \ref{HessianWeak1} can be reduced to $H_P(L_r,L_r)(p_0)\neq0$ by the following argument.
By Lemma \ref{BPscale}, $Z(B_X)=Z(B_{LX})$. By linearity, $$H_{1+LX}(L_r,L_r)(p)=LH_{1+X}(L_r,L_r)(p).$$
Replacing $X$ by $-X$ if necessary, we may assume that $H_{1+X}(L_r,L_r)(p_0)>0$. This changes the interval of $L$ which produce plurisubharmonic function at points lying in the normal direction from $(-\infty,L_0)$ if $H_{1+X}(L_r,L_r)(p_0)>0$ to $(-L_0,\infty)$ if $H_{1+X}(L_r,L_r)(p_0)<0$. Call the appropriate interval $I_L$.

Combining the results at strongly pseudoconvex points and weakly pseudoconvex points the following is a sufficient condition:
\begin{theorem}\label{TheoremSuff}
Suppose there exists a real valued function $X$ and $P=1+X$ such that:
\begin{enumerate} 
\item $W\cap U\subset Z(B_P)$, i.e., for all $p\in W\cap U$ $L_r(P)(p)=0$,
\item $H_P(L_r,L_r)(p_0)\neq0$ with $I_L$ as the interval of $L$ values for which $\rho$ is plurisubharmonic at the points lying in the normal direction from weakly pseudoconvex points, and 
\item there exists $L\in I_L$ such that $-\sh_{r(1+LX)}|_{b\Omega}\leq C\sL_r$ for some constant $C$.
\end{enumerate}
Then there exists $K>0$ and $L\in I_L$ such that $\rho=r(1+Kr+LX)$ is plurisubharmonic in a neighborhood of $p_0$.
\end{theorem}

\begin{proof}

Since $H_P(L_r,L_r)(p_0)\neq 0$, Propositon \ref{Suff} gives an interval $I_L$ and $K_0$, such that for all $L\in I_L$ and $K\geq K_0$ $\rho=r(1+Kr+LX)$ is plurisubharmonic at points lying on a normal line from weakly psuedoconvex points in some neighborhood of $p_0$. 

Assumption (3) and Theorem \ref{BoundaryPSH} and guarantees that for some $L\in I_L$ and some $K$ big enough $\rho$ is strictly plurisubharmonic at strongly psedoconvex points in some neighborhood of $p_0$. Notice that increasing the value of $K$ if necessary does not destroy plurisubharmonicity at the points lying on a normal line from weakly pseudoconvex points. Since $\rho$ is strictly plurisubharmonic at strongly pseudoconvex points in some neighborhood of $p_0$, Corollary \ref{C:S1} guarantees there is a neighborhood of $p_0$ such that $\rho$ is plurisubharmonic at  every point lying on a normal line from strongly pseudoconvex points.

Therefore there is a neighborhood of $p_0$ such that $\rho$ is plurisubharmonic at every point as desired.

\end{proof}

\section{An Example }\label{S:Examples}

In this section the Theorem \ref{TheoremSuff} is demonstarted in action.
\\Let $$\Omega=\{r(z,w)=\re(w)+|w|^2+\re(w)|z|^2+|z|^2|w|^2+|z|^4+|z|^6<0\}\subset \C^2,$$ and $p_0=(0,0)$. The Levi form is given by $$\sL_r(z,w)=(1+|z|^2)|z|^2\left(4|\frac12 +w|^2+|z|^4\right).$$ Therefore, $\Omega$ is pseudoconvex and strongly pseudoconvex for boundary points $(z,w)$ with $z\neq0$, that is, $$W=b\Omega\cap \{z=0\}.$$
The complex Hessian $$\sh_r(z,w)=\begin{pmatrix} \re(w)+|w|^2+4|z|^2+9|z|^4 & \frac12z+z\bar w\\
\frac12\bar z+\bar zw&1+|z|^2 
\end{pmatrix}= \re w+|w|^2+\frac{15}{16}|z|^2+13|z|^4+9|z|^6$$ is positive semi-definite on the boundary
as  $$\re(w)+|w|^2=\frac{-|z|^4-|z|^6}{1+|z|^2}=-|z|^4$$ when restricted to the boundary. 
However, for any weakly pseudoconvex points $p=(0,w)\in W$, $$\sh_r(p)=\re(w)+|w|^2,$$ showing that $r$ cannot be plurisubharmonic in any neighborhood of $p_0$.

Let $$\rho=r(1+Kr+L|z|^2).$$ In the language of this paper $X=L|z|^2$ with $X_w=0$ and $X_z=L\bar z$.
\\The following computations show that all the hypothesis of Theorem \ref{TheoremSuff} are satisfied:
\\
\begin{enumerate}
\item $Z(B_P)=Z(|r_wX_{\bar z}|)=Z(|L||r_w||\bar z|)\supset \{z=0\}$, therefore $W\subset Z(B_P)$.\\
\item $H_P(L_r,L_r)(p_0)=L|r_w(p_0)|^2=\frac12L<0$ for $L<0$. Therefore $I_L=(-\infty,0)$.\\
\item for $L=-1\in I_L$, $$rP=r(1-|z|^2)=\re(w)+|w|^2+|z|^4-|z|^4\re(w)-|z|^4|w|^2-|z|^8$$ and $$\sh_{rP}=\begin{pmatrix}
4|z|^2-4|z|^2\re(w)-4|z|^2|w|^2-16|z|^6& -\bar z|z|^2-2\bar z|z|^2w\\
-z|z|^2-2z|z|^2\bar w&1-|z|^4
\end{pmatrix}=4|z|^2+o(|z|^2)=\mathcal O(\sL_r)$$ in a neighborhood of $p_0$.
\end{enumerate}
Therefore, by Theorem \ref{TheoremSuff} there exists $K\geq0$ such that  $\rho=r(1+Kr-|z|^2)$ is plurisubharmonic in some neighborhood of the origin. In fact, a direct computation of $\sh_\rho$ shows that $\rho$ is plurisubharmonic for any $K\geq0$.

\section{Higher order Taylor's formula}\label{S:HOTS}

Proposition \ref{Suff} gives a sufficient condition for $\rho=r(1+Kr+X)$ to be a local plurisubharmonic defining function. However it may be the case that $W\subset Z(B_X)$ and $H_X(L_r,L_r)(p)=0$ for all choices of $X$ and $p\in W$. The proposition then does not apply, but it is still possible that a plurisubharmonic $\rho$ may be constructed locally. To determine if that is the case, Taylor analysis to higher order needs to be considered. Each additional degree in the Taylor expansion imposes a new necessary condition, akin to Lemma  \ref{P:NWeak} and Lemma \ref{BPbound}.

From now on assume that $r$ is a real-analytic defining function for $\Omega$ and $r$ is plurisubharmonic on the boundary.
We introduce new notation to help us organize the calculations involving higher order Taylor approximations.
Denote by $A_kf$ the coefficient function of $r^k$ in the Taylor formula. That is,
\begin{align}\label{TaylorK}
f(q)=&\sum_{k=0}^\infty A_k(f)\frac{(-d_{b\Omega}(q))^k}{|\partial r(q)|^k} \notag\\
=&\sum_{k=0}^\infty A_k(f) u^k(q)r^k(q)
\end{align}
In Section \ref{S:Prel}, the first few $A_i(f)$'s were computed $A_0(f)=f(p)$, $A_1(f)= \re[N_rf](p)$, $A_2(f)=2\re[(N_rN_r)f](p)+\bar N_rN_rf(p)$, and so on. Also set $A_{-2}=A_{-1}=0$.

Equation \eqref{Hessian2} says
\begin{align}\label{Hessian4}
\sh_\rho=&\left(2KPH_r(L_r,L_r)+P^2\sh_r+2P\re[H_r(L_r,L_P)]+B_P\right)\notag\\
&+ r\bigg(4K^2H_r(L_r,L_r)+PQ_P+2\re[H_P(L_r,L_P)]+4KP\sh_r+4K\re[H_r(L_r,L_P)]\notag\\
&\qquad+2KH_P(L_r,L_r)\bigg)+r^2\left(4K^2\sh_r+\sh_P+2KQ_P\right).
\end{align}

Applying the Taylor expansion \eqref{TaylorK} to relevant terms in \eqref{Hessian4} and regrouping them according to the power of $r$
\begin{align}\label{TE1}
\sh_\rho(q)= \sum_{k=0}^\infty  r^k\bigg(& 2KPu^k A_k(H_r(L_r,L_r)) + P^2u^kA_k(\sh_r)+2Pu^kA_k(\re[H_r(L_r,L_P)]) \notag\\
&+u^kA_k(B_P)+4K^2u^{k-1}A_{k-1}(H_r(L_r,L_r))+Pu^{k-1}A_{k-1}(Q_P) \notag\\
&+2u^{k-1}A_{k-1}(\re[H_P(L_r,L_P)])+4KPu^{k-1}A_{k-1}(\sh_r) \notag\\
&+4Ku^{k-1}A_{k-1}(\re[H_r(L_r,L_P)])+2Ku^{k-1}A_{k-1}(H_P(L_r,L_r)) \notag\\
& +4K^2u^{k-2}A_{k-2}(\sh_r)+u^{k-2}A_{k-2}(\sh_P)+2Ku^{k-2}A_{k-2}(Q_P)\bigg).
\end{align}
Combining the like powers of $K$ in \eqref{TE1}
\begin{align}\label{TEFinal}
\sh_{\rho}(q)=\sum_{k=0}^\infty r^k\Bigg(& \bigg(P^2u^kA_k(\sh_r)+2Pu^kA_k(\re[H_r(L_r,L_P)])+u^kA_k(B_P) \notag\\
&\quad +Pu^{k-1}A_{k-1}(Q_P)+2u^{k-1}A_{k-1}(\re[H_P(L_r,L_P)])+u^{k-2}A_{k-2}(\sh_P)\bigg) \notag\\
&+K \bigg(2Pu^kA_k(H_r(L_r,L_r))+4Pu^{k-1}A_{k-1}(\sh_r)\notag\\
&\qquad\quad +4u^{k-1}A_{k-1}(\re[H_r(L_r,L_P)])+2u^{k-1}A_{k-1}(H_P(L_r,L_r))\notag\\
&\qquad\quad +2u^{k-2}A_{k-2}(Q_P)\bigg)\notag\\
&+K^2 \bigg(4u^{k-1}A_{k-1}(H_r(L_r,L_r))+4u^{k-2}A_{k-2}(\sh_r)\bigg)\Bigg)
\end{align}
rewrite the coefficient of each power of $r$ as a polynomial of $K$
\begin{align}\label{TEpolyK}
\sh_{\rho}(q)=&\sum_{k=0}^\infty r^k\bigg( F^0_k+KF^1_k+K^2F^2_k\bigg) = \sum_{k=0}^\infty r^kG_k,
\end{align}
where $G_k$ depends on $K$ and $P$.

Notice that $$F_k^2=4u^{k-1}A_{k-1}(H_r(L_r,L_r))+4u^{k-2}A_{k-2}(\sh_r)$$ is independent of the choice of $P$, while $F_0^2,F_1^2=0$.
As $F_k^2$ is the leading coefficient of the polynomial $G_k$  this term will have a great effect the range of $K$ for which $\rho$ is plurisubharmonic.

Our goal is to produce a function $P=1+X$ such that there exists a neighborhood  $U$ of $p_0$ such that for all $q\in U$ with 
$\pi(q)=p\in U\cap b\Omega$ there exists a positive integer $N$ such that
\begin{enumerate}
\item for all $k<N$, $G_k(p)=0$, and
\item $(-1)^NG_N(p)>0$.
\end{enumerate}
Then in a sufficiently small neighborhood $U^\prime$ of the $p$, for all $q\in \Omega\cap U^\prime$ with $\pi(q)=p$: 
$$\sh_\rho(q)=\sum_{k=N}^\infty r^kG_k = r^NG_N+\mathcal O (r^{N+1})\geq r^N(G_N-\frac12 G_N)=\frac12G_N r^N>0.$$
If $G_k=0$ for all $k\in \N$, then $\sh_\rho(q)=0$ on the line normal to $p$ and $\rho$ is plurisubharmonic at those points.  

\begin{remark}
Note that $N$ need not be the same for all $p\in b\Omega\cap U$. In fact, this was observed in Section \ref{S:pshonboundary} when considering strongly pseudoconvex points and weakly pseudoconvex points separately. 
\end{remark}

Conversely: If for any neighborhood $U$ of $p_0$ there exists $p\in b\Omega$, $q\in \Omega$ with $\pi(q)=p\in U$ and a positive inter $N$, such that:
\begin{enumerate}
\item for all $k<N$, $G_k=0$, and
\item $(-1)^NG_N<0$
\end{enumerate}
then $\rho$ is not plurisubharmonic in any neighborhood of $p_0$.
In other words, the same $P$ needs to satisfy the assumptions for all points $p\in b\Omega\cap U$ in some neighborhood $U$ of $p_0$ for $\rho$ to be plurisubharmonic in $U$.

Putting these statements together we obtain the following
\begin{theorem}\label{MainTaylorN}
Let $\Omega=\{r<0\}\subset \C^2$ with a defining function $r$ which is plurisubharmonic on the boundary and let $U$ be a neighborhood of $p_0$. Then $\Omega$ admits a local plurisubharmonic defining function near $p_0$ if and only if there exists a real-valued function $P=1+X$ and $K\in \R$ such that for all $p\in b\Omega\cap U$, there exists $N\in\N$ such that 
\begin{enumerate}
\item for all $k<N$, $G_k(p)=0$ and $(-1)^NG_N(p)>0$, or
\item  $G_k=0$ for all $k\in \N$.
\end{enumerate} 
\end{theorem}

Each $G_k=0$ for $k=0,1,2,...,N-1$ imposes a necessary condition on $P=1+X$ and $K$.
In Section \ref{S:pshonboundary}, $G_0$ and $G_1$ for weakly pseudoconvex points were computed $$G_0=0 \text{ is equivalent to } W\subset (Z(B_P)) \text { and }$$
$$G_1=0 \text{ is equivalent to \eqref{HessianWeak1} vanishing to order $r^2$ },$$  
while $\rho$ remains strongly plurisubharmonic on the boundary at strongly pseudoconvex points. 

Since $G_k$ is a quadratic polynomial in $K$, the range of $K$ may be bounded above as well. Therefore $-\sh_{r(1+LX)}|_{b\Omega}\leq C\sL_r$ for some constant $C$ may not be sufficient condition for $\rho$ to plurisubharmonic on the boundary, as that may require to increase the value of $K$.

\bibliographystyle{acm}
\bibliography{BibMaster}

\begin{thebibliography}{10}

\bibitem{Behrens85}
{\sc Behrens, M.}
\newblock Plurisubharmonic defining functions of weakly pseudoconvex domains in
  {${\bf C}\sp 2$}.
\newblock {\em Math. Ann. 270}, 2 (1985), 285--296.

\bibitem{DAngelo82}
{\sc D'Angelo, J.~P.}
\newblock Real hypersurfaces, orders of contact, and applications.
\newblock {\em Ann. of Math. (2) 115}, no. 3 (1982), 615--637.

\bibitem{DAngelo_SCVRHS}
{\sc D'Angelo, J.~P.}
\newblock {\em Several complex variables and the geometry of real
  hypersurfaces}.
\newblock Studies in Advanced Mathematics. CRC Press, Boca Raton, FL, 1993.

\bibitem{DieFor77-1}
{\sc Diederich, K., and Forn{\ae}ss, J.~E.}
\newblock Pseudoconvex domains: an example with nontrivial {N}ebenh\"ulle.
\newblock {\em Math. Ann. 225}, 3 (1977), 275--292.

\bibitem{Fornaess79}
{\sc Forn{\ae}ss, J.~E.}
\newblock Plurisubharmonic defining functions.
\newblock {\em Pac. J. Math. 80}, 2 (1979), 381--388.

\bibitem{Gilbert91}
{\sc Gilbert, G.~T.}
\newblock Positive definite matrices and {S}ylvester's criterion.
\newblock {\em Amer. Math. Monthly 98}, 1 (1991), 44--46.

\bibitem{HerMcN09}
{\sc Herbig, A.-K., and McNeal, J.~D.}
\newblock Convex defining functions for convex domains.
\newblock {\em J. Geom. Anal. 22}, 2 (2012), 433--454.

\bibitem{HerMcN2012}
{\sc Herbig, A.-K., and McNeal, J.~D.}
\newblock Oka’s lemma, convexity, and intermediate positivity conditions.
\newblock {\em Ill. J. Math. 56}, 1 (2012), 195--211.

\bibitem{Kohn73}
{\sc Kohn, J.~J.}
\newblock Global regularity for $\overline \partial$ on weakly pseudo-convex
  manifolds.
\newblock {\em Trans. Amer. Math. Soc. 181\/} (1973), 273--292.

\bibitem{Liu19-1}
{\sc Liu, B.}
\newblock The {D}iederich-{F}ornaess index i: For domains of non-trivial index.
\newblock {\em Adv. Math. 353\/} (2019), 776 -- 801.

\bibitem{Liu19-2}
{\sc Liu, B.}
\newblock The {D}iederich-{F}ornaess index ii: For domains of trivial index.
\newblock {\em Adv. Math. 344\/} (2019), 289 -- 310.

\bibitem{Mernik20}
{\sc Mernik, L.}
\newblock Local plurisubharmonic defining functions on the boundary.
\newblock {\em Pac. J. Math. 307}, 1 (2020), 221–238.

\end{thebibliography}

\end{document}